\documentclass{amsproc}

 \newtheorem{theorem}{Theorem}[section]
 \newtheorem{corollary}[theorem]{Corollary}
 \newtheorem{lemma}[theorem]{Lemma}
 
 \newtheorem{proposition}[theorem]{Proposition}
 \theoremstyle{definition}
 \newtheorem{definition}[theorem]{Definition}
 \theoremstyle{remark}
 \newtheorem{remark}[theorem]{Remark}

\numberwithin{equation}{section}

\usepackage{color}
\usepackage{graphics}
\usepackage{graphicx}
\usepackage{epsfig}
\usepackage{amssymb}
\usepackage{verbatim}
\usepackage{graphics,amsmath,amsfonts,amssymb} 
\usepackage{epsfig,verbatim,multicol}

\begin{document}
\title[Pasting and Reversing over Vector Spaces]{Pasting and Reversing Operations over some Vector Spaces}
\author[P. Acosta-Hum\'anez]{Primitivo Acosta-Hum\'anez}

\address[P. Acosta-Hum\'anez]{Universidad del Norte, Barranquilla - Colombia}
\email{pacostahumanez@uninorte.edu.co}
\author[A. Chuquen]{Adriana Lorena Chuquen}
\address[A. Chuquen]{Universidad Sergio Arboleda, Bogot\'a - Colombia}
\author[A. Rodr\'\i guez]{\'Angela Mariette Rodr\'{\i}guez}
\address[A. Rodr\'\i guez]{Universidad Nacional de Colombia, Bogot\'a - Colombia}
\dedicatory{To Angie Marcela Acosta, with occasion of her 15th birthday}
\maketitle

\begin{abstract}
Pasting and Reversing operations have been used successfully over the set of integer numbers, simple permutations, rings and recently over a generalized vector product. In this paper, these operations are defined from a natural way to be applied over vector spaces. In particular we study Pasting and Reversing over vectors, matrices and we rewrite some properties for polynomials as vector space. Finally we study some properties of Reversing through linear transformations as for example an analysis of palindromic and antipalindromic vector subspaces.\\


\end{abstract}

\section*{Introduction}
\noindent\emph{Pasting} and \emph{Reversing} are natural processes that the people do day after day, we \emph{paste} two objects when we put them together
as one object, and we \emph{reverse} one object when we reflect it over a symmetry
axis. We can apply these processes over words, thus, \emph{Pasting} of \emph{lumber} with \emph{jack} is \emph{lumberjack}, while \emph{Reversing} of \emph{lumber} is \emph{rebmul}. A celebrate phrase of Albert Einstein is \emph{I prefer pi}, in which \emph{Reversing} of this phrase is itself and for instance this is a \emph{palindromic} phrase, another palindromic phrases can be found in  \texttt{http://www.palindromelist.net}.\medskip

Similarly to palindromic phrases, we can think in palindromic poetries, where each line can be palindromic or the hole poetry is palindromic.  The following poems can be found at \texttt{http://www.trauerfreuart.de/palindrome-poems.htm}\medskip
\begin{small}

\begin{center}
\textbf{Deific? A poem}\\
\emph{
Same deficit sale: doom mood. Elastic if edemas?\\
Loops secreting in a doom mood. An igniter: cesspool.\\
Set agony care till in a doom mood. An illiteracy: no gates.\\
Senile fileting: I am, God, doom mood. Dogma: ignite lifelines.\\
Straws? Send a snowfield in a doom mood. An idle, if won sadness warts.\\
Me, opacified.}
\end{center}

\begin{multicols}{2}
\begin{center}
{\textbf{Put in us - sun it up}}

\emph{Put in rubies, I won't be demandable.}

\emph{Balderdash: sure fire bottle fill-in.}

\emph{Raw, put in urn action, I'm odd.}

\emph{Local law: put in ruts. Awareness}

\emph{elates pure gnawed limekiln. Us:}

\emph{sunlike, mildew, anger, upset.}

\emph{A lessen era was: turn it up! Wall,}

\emph{a cold domino: it can run it up.}

\emph{Warn: ill I felt to be rife. Rush! Sad.}

\emph{Red label, bad name, debt nowise.}

\emph{I burn it up}
\end{center}
\columnbreak
\begin{center}
{\textbf{Space caps}}

\emph{Seed net: tabard. No citadel like sun is}

\emph{but spirit. Sense can embargo to get on.}

\emph{Still amiss: a pyro-memoir, an ecstasy.}

\emph{A detail, if fades, paler, tall, affined dusk.}

\emph{Row no risks, asks ironwork, sudden -}

\emph{if fall at relapsed affiliate - days at scenario:}

\emph{memory, pass! I'm all. It's not ego to grab}

\emph{menaces. Nest. I rip stubs in use, killed}

\emph{at icon. Drab attendees.}
\end{center}
\end{multicols}

\end{small}

One mathematical theory to express these processes as operations was developed by the first author in \cite{Ac1,Ac2}, followed recently by \cite{acarnu,acchro1, AM}.\\

\noindent In \cite{Ac2} is introduced the concept \emph{Pasting} of positive integers to obtain their squares as well their squares roots. Five years later, in \cite{Ac1}, are defined in a general way the concepts \emph{Pasting} and \emph{Reversing} to obtain genealogies of \emph{simple permutations} in the right block of \emph{Sarkovskii order} which contains the powers of two. Two years later, in \cite{acchro1}, were applied \emph{Pasting} and \emph{Reversing}, as well \emph{palindromicity} and \emph{antipalindromicity}, over the ring of polynomials, differential rings and mathematical games incoming from M. Tahan's book \emph{The man who counted}. Another approaches for \emph{reversed} polynomials, \emph{palindromic} polynomials and \emph{antipalindromic} polynomials can be found in \cite{Br, MaRa}. One year later, in \cite{acarnu}, is applied \emph{Reversing} over matrices to study a generalized vector product, in particular were studied relationships between \emph{palindromicity} and \emph{antipalindromicity} with such generalized vector product. Finally, in the preprint \cite{AM} were applied \emph{Pasting} and \emph{Reversing}
over \emph{simple permutations} with mixed order $4n+2$, following \cite{Ac1}.\\

\noindent The aim of this paper is to study \emph{Pasting} and \emph{Reversing}, as well \emph{palindromicity} and \emph{antipalindromicity}, over vector spaces (vectors, matrices, polynomials, etc.). Some properties are analyzed for vectors, matrices and polynomials as vector spaces, in particular we prove that $W_a\subset V$ (set of antipalindromic vectors of $V$) and $W_p\subset  V$ (set of palindromic
vectors of $V$) are vector subspaces of $V$ ($V$ is a vector space). Moreover, $V = W_a \oplus W_p$ and in
consequence $\dim(V) = \dim(W_a) + \dim(W_p)$.\\
 
 The reader does not need a high mathematical level to understand this paper, is enough with a basic knowledge of linear algebra and matrix theory, see for example the books given in references \cite{Ev,Lang}. Finally, as \emph{butterfly effect}, we hope that the results and approaches presented here can be used and implemented in the teaching of basic linear algebra for undergraduate level.

\section{Pasting and Reversing over Vectors}\label{sectionv}
\noindent In this section we study Pasting and Reversing over vectors in two different approaches, the first one corresponds to an analysis without linear transfomatios, using basic definitions and properties of vectors. We consider the field $K$ and the vector space $V=K^n$.
\subsection{A first study without linear transformations} Here we study Pasting and Reversing over vectors using the basic definitions and properties of vectors. In this way, any student beginner of linear algebra can understand the results presented. We start giving the definition of Reversing.
\begin{definition} \label{dd1}
 Let be $v=(v_{1},v_{2},\ldots,v_{n}) \in K^n$. Reversing of $v$, denoted by $\widetilde{v}$, is given by $\widetilde{v}=(v_{n},v_{n-1},\ldots,v_{1})$.
\end{definition}

Definition \ref{dd1} lead us to the following proposition.

\begin{proposition}\label{ppro1} Consider $v$ and $\widetilde{v}$ as in Definition \ref{dd1}. The following statements hold:
 \begin{enumerate}
\item $\tilde{\tilde{v}}=v$
\item $\widetilde{av+bw}=a\widetilde{v}+b\widetilde{w}$, being $a,b\in
K$ and $v,w\in V$
\item $v\cdot w=\widetilde{v}\cdot\widetilde{w}$
\item $\widetilde{(v\times w)}=\widetilde{w}\times\widetilde{v}$ for all $v,w\in K^{3}$
\end{enumerate}
\end{proposition}

\begin{proof} The proof is done according to each item:
\begin{enumerate}
\item Due to
$v\in K^n$, then by Definition \ref{dd1} we have that $\tilde{v}=(v_{n},v_{n-1},\ldots,v_{1})$ and for instance $\tilde{\tilde{v}}=(v_{1},v_{2},\ldots,v_{n})=v$.

\item Let $z=av+bw$, where $a,b\in K$ and $v,w\in V$. Therefore $$z=(av_1+bw_1,\ldots,av_n+bw_n)$$ and in consequence
$$\widetilde{av+bw}=\widetilde{z}=(av_n+bw_n,\ldots,av_1+bw_1).$$ By basic theory of linear algebra, particularly by properties of vectors, we have that $$\widetilde{z}=(av_{m},av_{m-1},\ldots,av_{1})+(bw_{m},bw_{m-1},\ldots,bw_{1}),$$ which implies that $$\widetilde{z}=a{(v_{n},v_{n-1},\ldots,v_{1})}+b(w_{n},w_{n-1},\ldots,w_{1})$$ and for instance $\widetilde{av+bw}=a\tilde{v}+b\tilde{w}$.
\item Let assume $v=(v_1,v_2,\ldots,v_n)$ and $w=(w_1,w_2,\ldots, w_n)$, thus $$v\cdot w=v_1w_1+\ldots+v_nw_n.$$ Now, $\widetilde{v}=(v_n,v_{n-1},\ldots, v_1)$ and $\widetilde{w}=(w_n,w_{n-1},\ldots, w_1)$, so we obtain $$\widetilde{v}\cdot \widetilde{w}=v_nw_n+v_{n-1}w_{n-1}+\ldots+ v_1w_1=v\cdot w.$$
\item Consider $v,w\in K^{3}$, where $v=(v_{1},v_{2},v_{3}),\,w=(w_{1},w_{2},w_{3})$. The vector product between $v$ and $w$ is given by $$v\times w =\left|\begin{matrix} e_1 & e_2 & e_3\\ v_1 & v_2 & v_3\\ w_1 & w_2 & w_3\end{matrix} \right|=\left(\left|\begin{matrix}v_2 & v_3\\ w_2 & w_3\end{matrix}\right|,-\left|\begin{matrix}v_1 & v_3\\ w_1 & w_3\end{matrix}\right|,\left|\begin{matrix}v_1 & v_2\\ w_1 & w_2\end{matrix}\right| \right),$$ by Definition \ref{dd1} we have that $$\widetilde{v\times w}=\left(\left|\begin{matrix}v_1 & v_2\\ w_1 & w_2\end{matrix}\right|,-\left|\begin{matrix}v_1 & v_3\\ w_1 & w_3\end{matrix}\right|,\left|\begin{matrix}v_2 & v_3\\ w_2 & w_3\end{matrix}\right| \right).$$ Now, by properties of determinants (interchanging rows and columns) we obtain
    $$\left(\left|\begin{matrix}w_2 & w_1\\ v_2 & v_1\end{matrix}\right|,-\left|\begin{matrix}w_3 & w_1\\ v_3 & v_1\end{matrix}\right|,\left|\begin{matrix}w_3 & w_2\\ v_3 & v_2\end{matrix}\right| \right)=\left|\begin{matrix} e_1 & e_2 & e_3\\ w_3 & w_2 & w_1\\ v_3 & v_2 & v_3\end{matrix} \right|,$$
    for instance $\widetilde{v\times w}=\widetilde{w}\times\widetilde{v}.$
\end{enumerate}
\end{proof}

Definition \ref{dd1} lead us to the following definition.

\begin{definition} \label{dd2} The vectors $v$ and $w$ are called palindromic vector
and antipalindromic vector respectively whether they satisfy $\widetilde{v}=v$ and $\widetilde{w}=-w$.
\end{definition}

Definition \ref{dd2} lead us to the following result.

\begin{proposition}\label{ppro2}
The following statements hold.
\begin{enumerate}
\item The sum of two palindromic vectors belonging to $K^n$ is a palindromic vector belonging to $K^n$.
\item The sum of two antipalindromic vectors belonging to $K^n$ is an antipalindromic vector belonging to $K^n$.
\item The vector product of two palindromic vectors belonging to $K^3$ is an antipalindromic vector belonging to $K^3$.
\item The vector product of two antipalindromic vectors belonging to $K^3$ is the vector $(0,0,0)$.
\item The vector product of one palindromic vector belonging to $K^3$ with one antipalindromic vector belonging to $K^3$ is a palindromic vector belonging to $K^3$.
\end{enumerate}
\end{proposition}

\begin{proof} We prove each statement according to its item.
\begin{enumerate}
\item Let $v$ and $w$ be palindromic vectors. We have that
$\widetilde{v+w}=\widetilde{v}+\widetilde{w}=v+w$. In consequence,
$v+w$ is a palindromic vector.
\item Let $v$ and $w$ be antipalindromic vectors. We have that
$$\widetilde{v+w}=\widetilde{v}+\widetilde{w}=-v-w=-(v+w).$$ In consequence, $v+w$ is an antipalindromic vector.
\item Let assume $v=(v_1,v_2,v_1)$ and  $w=(w_1,w_2,w_1)$. The vector product between $v$ and $w$ is
$$v\times w=(v_{2}w_{1}-w_{2}v_{1},0,v_{1}w_{2}-v_{2}w_{1})=z.$$ Now, by Definition \ref{dd1} we obtain $$\widetilde{z}=(v_{1}w_{2}-v_{2}w-{1},0,v_{2}w_{1}-v_{1}w_{2})$$ and $$-\widetilde{z}=(v_{2}w_{1}-v_{1}w_{2},0,v_{1}w_{2}-v_{2}w_{1}).$$ For instance $z=-\tilde{z}$ and owing to Definition \ref{dd2}, $v\times w$ is an antipalindromic vector.
\item Let assume $v=(-v_{1},0,v_{1})$ and $w=(-w_{1},0,w_{1})$, then the vector product between $v$ and $w$ is given by $v\times w=(0,0,0)$.
\item Consider $v=(v_1 ,v_2, v_1)$ and $w=(-w_1 ,0,w_1)$, the vector product between $v$ and $w$ is given by $v\times w=(v_2 w_1,-2v_1 w_1,v_2 w_1)$. Now, denoting $z:=v\times w$, we have by Definition \ref{dd1} that $\widetilde{z}=(v_2 w_1,-2v_1 w_1,v_2 w_1)$. Thus, by Definition \ref{dd2} we obtain that $v\times w$ is a palindromic vector.
\end{enumerate}
\end{proof}
\begin{remark}\label{remacarnu}
In \cite{acarnu} were studied, in a more general way, the vector product for palindromic and antipalindromic vectors. There was used a generalized vector product and were obtained some results involving the palindromicity and antipalindromicity of vectors. For completeness we present in section \ref{sectionm} such results with proofs in detail.
\end{remark}

Now we proceed to introduce the concept of Pasting over vectors.

\begin{definition} \label{dd3}
Consider $v\in K^{n}$ and $w\in K^m$, then
$v\diamond w$ is given by $$(v_{1},v_{2},\ldots,v_{n})\diamond(w_{1},w_{2},\ldots,w_{m})=(v_{1},v_{2},\ldots,v_{n},w_{1},w_{2},\ldots,w_{m})$$
\end{definition}

The following properties are consequences of Definition \ref{dd3}.

\begin{proposition}\label{ppro3}
If $V=K^n$ and $W=K^m $, then $V \diamond W\cong K^{n+m}$
\end{proposition}

\begin{proof} Let $B_{n}=\{b_{1},b_{2},\ldots,b_{n}\}$ and $B_{m}=\{c_{1},c_{2},\ldots,c_{m}\}$ basis of $K^n$ and $K^m$ respectively. Due to
$v\in K^n$ and $w\in K^n$, we have by Definition \ref{dd3} that
$v\diamond w \in K^{n+m}$, then there exists a basis $B_{n+m}=\{d_{1},d_{2},\ldots,d_{n+m}\}$ belonging to $K^{n+m}$, for instance $K^n\diamond K^m  \cong K^{n+m}$.
\end{proof}

As immediate consequence of Proposition \ref{ppro3} we obtain the following result.

\begin{corollary}\label{ccol1}
$dim(V\diamond W)$ $=$ $dimV+dimW$
\end{corollary}

\begin{proposition}\label{pprop4} The following statements holds.
\begin{enumerate}
\item$\widetilde{v\diamond w}=\tilde{w}\diamond \tilde{v}$
\item$(v\diamond w)\diamond z=v\diamond (w\diamond z)$
\end{enumerate}
\end{proposition}

\begin{proof}
We consider each item separately.
\begin{enumerate}
\item Consider $V=K^n$ and $W=K^m$. Suppose that $v=(v_{1},v_{2},\ldots,v_{n})\in V$ and  $w=(w_{1},w_{2},\ldots,w_{m})\in W$. Owing to Definition \ref{dd3} we have that $$v \diamond w=(v_{1},v_{2},\ldots,v_{n},w_{1},w_{2},\ldots,w_{m})=z.$$ Now, by Definition \ref{dd1} we obtain $$\widetilde{z}=(w_{m},w_{m-1},\ldots,w_{1},v_{n},v_{n-1},\ldots,v_{1}),$$ which by Definition \ref{dd3} lead us to
$$\widetilde{z}=(w_{m},w_{m-1},\ldots,w_{1})\diamond(v_{n},v_{n-1},\ldots,v_{1})$$ and therefore
$$\widetilde{v\diamond w}=\tilde{w}\diamond \tilde{v}.$$
\item Let assume
$V=K^n$, $W=K^m$, $Z=K^{\ell}$,$v\in V$, $w\in W$ and $z\in Z$. By Definition \ref{dd3} we have that $$(v\diamond w)\diamond
z=(v_{1},v_{2},\ldots,v_{n},w_{1},w_{2},\ldots,w_{m})\diamond(z_{1},z_{2},\ldots,z_{\ell}),$$ which implies that
$$(v\diamond w)\diamond
z=(v_{1},v_{2},\ldots,v_{n},w_{1},w_{2},\ldots,w_{m},z_{1},z_{2},\ldots,z_{\ell}).$$ Again, in virtue of Definition \ref{dd3}, we have that $$(v\diamond w)\diamond
z=(v_{1},v_{2},\ldots,v_{n})\diamond((w_{1},w_{2},\ldots,w_{m})\diamond(z_{1},z_{2},\ldots,z_{\ell})),$$ thus we can conclude that $$(v\diamond w)\diamond z=v\diamond(w\diamond z).$$
\end{enumerate}
\end{proof}

\begin{proposition}\label{tt1} Let $V=K^n$ be a vector space. Consider $W_{p}$ and $W_{a}$ as the sets of palindromic and antipalindromic vectors of $V$ respectively. The following statements hold.
\begin{enumerate}
\item $W_{p}$ is a vector subspace of $V$,
\item $\dim W_{p}=\lceil\frac{n}{2}\rceil$,
\item $W_{a}$ is a vector subspace of $V$,
\item $\dim W_{a}=\lfloor\frac{n}{2}\rfloor$,
\item $V=W_p\oplus W_a$,
\item $\forall v\in V$, $\exists (w_p,w_a)\in W_p\times W_a$ such that $v=w_p+w_a$.
\end{enumerate}
\end{proposition}

\begin{proof} We consider each statement separately.
\begin{enumerate}
\item Suppose that $a,b\in K$ and $v,w\in W_p$. For instance we have that $v=\tilde{v}$ and $w=\tilde{w}$. By Proposition \ref{ppro1} we observe that $\widetilde{av+bw}=a\tilde{v}+b\tilde{w}=av+bw\in W_p$, in consequence, $W_p$ is a vector subspace of $V$.
\item We analyze the cases when $n$ is even and also when $n$ is odd.
\begin{enumerate}
\item Consider $V= K^n$ and we start assuming that $n=2k$. If $v\in W_{p}$, then
$$(v_{1},v_{2},\ldots,v_{2k-1},v_{2k})=(v_{2k},v_{2k-1},\ldots,v_{2},v_{1}),$$
for instance, we have that
$$\begin{array}{lll}
v_{1}&=&v_{2k}\\
v_{2}&=&v_{2k-1}\\
&\vdots&\\
v_{k}&=&v_{k+1},
\end{array}$$
which lead us to $v=(v_{1},v_{2},\ldots,v_{k},v_{k},\ldots,v_{2},v_{1})$. In this way we write the vector $v$ as follows: $$v=v_{1}(1,0,\ldots,0,0,\ldots,1)+\ldots+v_{k}(0,0,\ldots,1,1,\ldots,0,0),\quad v_i\in K.$$ The set of vectors of the previous linear combination are palindromic and linearly independent vectors, for instance they are a basis for $W_p$ and in consequence $$\dim W_{p}=k=\left\lceil\frac{2k}{2}\right\rceil=\left\lceil\frac{n}{2}\right\rceil.$$
\item Consider $V=K^n$ and now we assume that $n=2k-1$. If $v\in W_{p}$, then
$$(v_{1},v_{2},\ldots,v_{2k-2},v_{2k-1})=(v_{2k-1},v_{2k-2},\ldots,v_{2},v_{1}).$$ Thus, we have that
$$\begin{array}{lll}
v_{1}&=&v_{2k-1}\\
v_{2}&=&v_{2k-2}\\
&\vdots&\\
v_{k-1}&=&v_{k+1},
\end{array}$$
that is, $k-1$ pairs plus the fixed component $v_k$. This lead us to express the vector $v$ as follows $$v=(v_{1},v_{2},\ldots,v_{k-1},v_{k},v_{k-1},\ldots,v_{2},v_{1})$$ and for instance we have that
$$v=v_{1}(1,0,\ldots,0,0,\ldots,1)+\ldots+v_{k}(0,0,\ldots,0,1,0,\ldots,0,0),\quad v_i\in K.$$ The set of vectors of the previous linear combination are palindromic and linearly independent vectors, for instance they are a basis for $W_p$ and in consequence $$\dim W_{p}=k=\left\lceil\frac{2k-1}{2}\right\rceil=\left\lceil\frac{n}{2}\right\rceil.$$
\end{enumerate}
In this way, we have proven that for all $n\in \mathbb{Z}^+$, $\dim W_{p}=\left\lceil\frac{n}{2}\right\rceil$.

\item Suppose that $a,b\in K$ and $v,w\in W_a$. For instance we have that $\tilde{v}=-v$ and $\tilde{w}=-w$. By Proposition \ref{ppro1} we can observe that $$\widetilde{av+bw}=a\tilde{v}+b\tilde{w}=-(av+bw)$$ and for instance $av+bw\in W_a$, which implies that $W_a$ is a vector subspace of $V$.
\item We analyse the cases when $n$ is even as well when $n$ is odd.
\begin{enumerate}
\item Consider $V= K^n$ and we can start assuming that $n=2k$. If $v\in W_{a}$, then
$$(v_{1},v_{2},\ldots,v_{2k-1},v_{2k})=-(v_{2k},v_{2k-1},\ldots,v_{2},v_{1}),$$
for instance, we have that
$$\begin{array}{lll}
v_{1}&=&-v_{2k}\\
v_{2}&=&-v_{2k-1}\\
&\vdots&\\
v_{k}&=&-v_{k+1},
\end{array}$$
which lead us to $$v=(v_{1},v_{2},\ldots,v_{k},-v_{k},\ldots,-v_{2},-v_{1}).$$ This implies that $$v=v_{1}(1,0,\ldots,0,0,\ldots,-1)+\ldots+v_{k}(0,0,\ldots,1,-1,\ldots,0,0),\quad v_i\in K.$$ The set of vectors of the previous linear combination are antipalindromic and linearly independent vectors, for instance they are a basis for $W_a$ and in consequence $$\dim W_{a}=k=\left\lfloor\frac{2k}{2}\right\rfloor=\left\lfloor\frac{n}{2}\right\rfloor.$$
\item Consider $V=K^n$ and now we suppose that $n=2k-1$. If $v\in W_{a}$, then
$$(v_{1},v_{2},\ldots,v_{2k-2},v_{2k-1})=-(v_{2k-1},v_{2k-2},\ldots,v_{2},v_{1}).$$ Thus, we obtain that
$$\begin{array}{lll}
v_{1}&=&-v_{2k-1}\\
v_{2}&=&-v_{2k-2}\\
&\vdots&\\
v_{k-1}&=&-v_{k+1},
\end{array}$$
that is, $k-1$ pairs plus the fixed component $v_k=0$.
This lead us to express the vector $v$ as follows: $$v=(v_{1},v_{2},\ldots,v_{k-1},0,-v_{k-1},\ldots,-v_{2},-v_{1})$$ and for instance we have that
$$v=v_{1}(1,0,\ldots,0,0,\ldots,-1)+\ldots+v_{k-1}(0,0,\ldots,1,0,-1,\ldots,0,0),\quad v_i\in K.$$ The set of vectors of the previous linear combination are antipalindromic and linearly independent vectors, for instance they are a basis for $W_a$ and in consequence $$\dim W_{a}=k-1=\left\lfloor\frac{2k-1}{2}\right\rfloor=\left\lfloor\frac{n}{2}\right\rfloor.$$
\end{enumerate}
In this way we have proven that for all $n\in \mathbb{Z}^+$, $\dim W_{a}=\left\lfloor\frac{n}{2}\right\rfloor$.

\item Due to $W_p\cap W_a=\{\mathbf{0}\}$ and $\dim W_p+\dim W_a = \left\lceil \frac{n}{2}\right\rceil + \left\lfloor \frac{n}{2}\right\rfloor=n=\dim V$, therefore $W_p+ W_a=V$ and in consequence $W_p\oplus W_a=V$.
\item Consider $v\in V$, we can observe that $$w_p=\frac{v+\widetilde v}{2}$$ is a palindromic vector. In the same way we can observe that $$w_a=\frac{v-\widetilde v}{2}$$ is an antipalindromic vector and for instance $v=w_p+w_a,\, \forall v\in V$.
\end{enumerate}
\end{proof}

\subsection{Reversing as linear transformation}
\noindent We denote by $\mathcal R$ the transformation Reversing, i.e., for all $v\in K^n$, $\mathcal Rv=\widetilde v$. Thus, we summarize some previous results in the next proposition.
\begin{proposition}\label{proplint}
The following statements hold
\begin{enumerate}
\item  $\mathcal R$ is a linear transformation.
\item The transformation matrix of $\mathcal R$ is given by  \begin{displaymath}\widetilde I_n=\begin{pmatrix}0&0&\cdots&0&1\\0&0&\cdots&1&0\\\vdots\\0&1&\cdots&0&0\\1&0&\cdots&0&0\end{pmatrix}\in\mathcal M_{n\times n}(K).\end{displaymath}
\item $\mathcal R$ is an automorphism of $K^n$.
\item Minimal and characteristic polynomial of $\mathcal{R}$ are given respectively by $$Q(\lambda)=\lambda^2-1,\quad P(\lambda)=\displaystyle{ \left\{ { (\lambda+1)^m(\lambda-1)^m\,for\,n=2m \atop (\lambda+1)^m(\lambda-1)^{m+1}\,for\,n=2m+1         } \right.      }.$$
\item $K^n=\ker(\mathcal R-id)\oplus \ker(\mathcal R+id)$.
\item $\dim \ker(\mathcal R-id)=\left\lceil\frac{n}{2}\right\rceil$, $\dim \ker(\mathcal R+id)=\left\lfloor\frac{n}{2}\right\rfloor$.
\item $\forall v\in K^n$ let $\mathcal F_p(v)$ be a palindromic vector and let $\mathcal F_a(v)$ be an antipalindromic vector. Then $\mathcal F_p$ and $\mathcal F_a$ (\textbf{Palindromicing and Antipalindromicing transformations} respectively) are isomorphisms from $K^n$ to $\ker(\mathcal R-id)$ and from $K^n$ to $\ker(\mathcal R+id)$ respectively. Furthermore, $\forall v\in K^n$, $v=\mathcal F_p(v)+\mathcal F_a(v)$.
 \end{enumerate}
\end{proposition}
\begin{proof} We proceed according to each item.
\begin{enumerate}
\item For any scalar $\alpha\in K$ and any pair of vectors $v,w\in K^n$ we obtain
$\mathcal R(v+w)=\widetilde{v+w}=\widetilde{v}+\widetilde{w}=\mathcal{R}v+\mathcal{R}w$ and $\mathcal R(\alpha v)=\widetilde{\alpha v}=\alpha\widetilde{v}=\alpha\mathcal{R}v$.
\item By definition of Reversing we have
$\mathcal{R}:K^n\rightarrow K^n$ is such that $$\mathcal R(v_1,v_2,\ldots, v_{n-1},v_n)=(v_1,v_2,\ldots, v_{n-1},v_n)\widetilde I_n,$$
being $\widetilde I_n$ and $I_n$ (\emph{Reversing of the identity matrix} and the identity matrix of size $n\times n$) given by \begin{displaymath} \widetilde I_n=\begin{pmatrix}0&0&\cdots&0&1\\0&0&\cdots&1&0\\\vdots\\0&1&\cdots&0&0\\1&0&\cdots&0&0\end{pmatrix}\in\mathcal M_{n\times n}(K),\quad I_n=\begin{pmatrix}1&0&\cdots&0&0\\0&1&\cdots&0&0\\\vdots\\0&0&\cdots&1&0\\0&0&\cdots&0&1\end{pmatrix}\in\mathcal M_{n\times n}(K).\end{displaymath}
\item Owing to $\mathcal{R}:K^n\rightarrow K^n$ and $\widetilde{\widetilde v}=v$ implies that $\mathcal R$ is left-right invertible, then $\mathcal R$ is an automorphism of $K^n$.
\item We observe that $\mathcal R^2-id=0$, being $id$ the identity tranformation, therefore $Q(\lambda)=\lambda^2-1$ is the minimal polynomial of $\mathcal{R}$, that is, $Q(\widetilde I_n)=\mathbf 0_n\in\mathcal M_{n\times m}$ ($\mathbf 0_n$ is the zero matrix of size $n\times n$). For instance, $Q$ divides to the characteristic polynomial of  $\mathcal R$ which is given by
\begin{displaymath}P(\lambda)=\det (\widetilde I_n-\lambda I_n)=\displaystyle{ \left\{ { (\lambda+1)^m(\lambda-1)^m\,for\,n=2m \atop (\lambda+1)^m(\lambda-1)^{m+1}\,for\,n=2m+1         } \right.      },\quad Q\mid P.\end{displaymath}
\item Due to $\mathcal R^2-id=(\mathcal R-id)(\mathcal R+id)$, then $K^n=\ker(\mathcal R-id)\oplus \ker(\mathcal R+id)$.
\item  By definition of $\lfloor\frac{n}{2}\rfloor$ and $\lceil\frac{n}{2}\rceil$ we see that $(\lambda-1)^{\lceil\frac{n}{2}\rceil}(\lambda+1)^{\lfloor\frac{n}{2}\rfloor}=P(\lambda)$. Now, due to $K^n=\ker(\mathcal R-id)\oplus \ker(\mathcal R+id)$, then $\dim \ker(\mathcal R-id)=\lceil\frac{n}{2}\rceil$, $\dim \ker(\mathcal R+id)=\lfloor\frac{n}{2}\rfloor$ and $\lceil\frac{n}{2}\rceil+\lfloor\frac{n}{2}\rfloor=n$.
\item Palindromicing and Antipalindromicing transformations, denoted by $\mathcal F_p$ and $\mathcal F_a$ respectively, are given by
\begin{displaymath}\begin{array}{lllllllllllll}
 \mathcal F_p&:& &K^n&\rightarrow&\ker(\mathcal R-id)&&& \mathcal F_a&:& K^n&\rightarrow&\ker(\mathcal R+id) \\ &&&&&&,&&&&&&\\&&& v&\mapsto&\frac{v+\mathcal R(v)}{2}&&&&& v&\mapsto&\frac{v-\mathcal R(v)}{2}\end{array}.
 \end{displaymath}
 Due to $\mathcal R$ and $id$ are linear transformations defined over $K^n$, then $$\mathcal F_p=\frac{\mathcal R-id}{2}\quad \textrm{and} \quad \mathcal F_a=\frac{\mathcal R+id}{2}$$ are linear transformations over $K^n$. Owing to $\ker(\mathcal F_p)=\ker(\mathcal F_a)=\textbf{0}\in K^n$, then $\mathcal F_p$ and $\mathcal F_a$ are monomorphisms. In the same way, due to $\mathrm{Im}(\mathcal F_p)=\ker(\mathcal R-id)$ and $\mathrm{Im}(\mathcal F_a)=\ker(\mathcal R+id)$, then $\mathcal F_p$ and $\mathcal F_a$ are epimorphisms. For instance, $\mathcal F_p$ and $\mathcal F_a$ are isomorphisms from $K^n$ to $\ker(\mathcal R-id)$ and from $K^n$ to $\ker(\mathcal R+id)$ respectively. Finally, we can see that any vector $v\in K^n$ can be expressed as the sum of a palindromic vector (obtained through $\mathcal F_p$) with an antipalindromic vector (obtained through $\mathcal F_a$). That is, $\mathcal F_p(v)$ is a palindromic vector and $\mathcal F_a(v)$ is an antipalindromic vector, furthermore, $v=\mathcal F_p(v)+\mathcal F_a(v)$.
\end{enumerate}
\end{proof}

 \begin{remark} can see $\mathcal{R}$ as a particular case of a linear transformation associated to a \emph{permutation matrix}. Recall that $A_{\sigma}$ is a permutation matrix, defined over a given $\sigma\in S_n$, whether its associated linear transformation $\mathcal{R}_{\sigma}$ is given by $$\mathcal{R}_{\sigma}\,:\begin{array}{ccc}
 K^n&\longrightarrow& K^n\\(v_1,\ldots,v_n)&\mapsto&(v_{\sigma(1)},\ldots,v_{\sigma(n)}).
 \end{array}
 $$
 This means that Reversing corresponds to $\mathcal R_\sigma=\mathcal R$, where the permutation $\sigma$ and the matrix $A_\sigma$ are given respectively by $$\sigma=\begin{pmatrix}
 1&2&3&\ldots&n-1&n\\n&n-1&n-2&\ldots&2&1
 \end{pmatrix}, \quad A_\sigma=\begin{pmatrix}
 0&0&0&\ldots&0&1\\0&0&0&\ldots&1&0\\\vdots\\0&1&0&\ldots&0&0\\1&0&0&\ldots&0&0
 \end{pmatrix}.$$
 \end{remark}

\section{Pasting and Reversing over Polynomials}\label{sectionp}
\noindent In \cite{acchro1} we studied Pasting and Reversing over polynomials from an different approach, we studied these operations focusing on the ring structure for polynomials. In this section we rewrite some properties of Pasting and Reversing over polynomials, but considering to the polynomials as a vector space. Thus, we apply the previous results for vectors, which we gave in Section \ref{sectionv}.\\

\noindent Along this section we consider $(K_{n}[x],+,\cdot)$ as the vector space of the polynomials of degree less than or equal to $n$ over the field $K$. This vector space is isomorphic to $(K^{n+1},+,\cdot)$. In this context we do not impose conditions over the polynomials just like the conditions given in \cite{acchro1}, for example, we do not need that $x\nmid P(x)$. The following result summarizes the properties given in Section \ref{sectionv} for polynomials.
\begin{proposition}
Consider $P\in K_n[x]$, $Q\in K_m[x]$ and $R\in K_s[x]$, the following statements hold.

\begin{enumerate}
\item $\tilde{\tilde{P}}=P$
\item$\widetilde{P\diamond Q}=\tilde{Q}\diamond \tilde{P}$
\item$(P\diamond Q)\diamond R=P\diamond (Q\diamond R)$
\item $\widetilde{aP+bQ}=a\widetilde{P}+b\widetilde{Q}$, being $a,b\in
K$ and $P,Q\in K_{n}[x]$
\item The sum of two palindromic polynomials of degree $n$ is a palindromic polynomial of degree $n$.
\item The sum of two antipalindromic polynomials of degree $n$ is an antipalindromic polynomial of degree $n$.
\item If $V=K_n[x]$ and $W=K_m[x]$, then $V \diamond W= K_{n+m+1}[x]$.
\item $W_{p}$ is vector subspace of $K_n[x]$, being $W_{p}$ the set of palindromic polynomials of degree $K_n[x]$.
\item $\dim W_{p}=\lceil\frac{n+1}{2}\rceil$,
\item $W_{a}$ is a vector subspace of $K_n[x]$, being $W_{a}$ the set of antipalindromic polynomials of $K_n[x]$
\item $\dim W_{a}=\lfloor\frac{n+1}{2}\rfloor$.
\item $K_n[x]=W_p\oplus W_a$.
\item $\forall P\in K_n[x]$, $\exists (Q_p,Q_a)\in W_p\times W_a$ such that $P=Q_p+Q_a$.
\end{enumerate}
\end{proposition}
\begin{proof} Owing to $(K_n[x],+,\cdot)\cong (K^{n+1},+,\cdot)$ as vector spaces we apply $\mathcal R$ and $\diamond$ over the polynomials as vectors. The proof is done using Proposition \ref{proplint} and properties of Pasting proven in Section 1.1.

\end{proof}
\begin{remark}
 As we can see, this section is a rewriting of Section \ref{sectionv} without new results for polynomials as vector space, only we suggest the proofs based on the definition and properties of $\mathcal R$. Another interesting thing of this section is that we recover some results given in \cite{acchro1,MaRa}.
 \end{remark}

\section{Pasting and Reversing over Matrices}\label{sectionm}
\noindent In this section we consider the vector space $\mathcal M_{n\times m}$ (matrices of size $n\times m$ with elements belonging to $K$) which is isomorphic to $K^{nm}$. We present here different approaches for Pasting and Reversing.
\subsection{Pasting and Reversing by rows or columns}

We can see any matrix as a row vector of its column vectors as well as a column vector of its row vectors. Thus, to matrices we can introduce \emph{Pasting and Reversing} by rows and columns respectively. Let us denote by $\widetilde A_r$ Reversing of the row vectors $v_i\in K^n$ of $A$ and by $\widetilde A_c$ Reversing of the column vectors $c_j\in K^m$ of $A$, where $1\leq i\leq m$ and $1\leq j\leq n$. For instance,

\begin{displaymath}
\widetilde{A}_r=\begin{pmatrix}\widetilde{v_{1}}\\ \widetilde{v_{2}}\\ \vdots\\ \widetilde{v_{n}}\end{pmatrix}=\begin{pmatrix}\mathcal R{v_{1}}\\ \mathcal R{v_{2}}\\ \vdots\\ \mathcal R{v_{n}}\end{pmatrix},\quad\widetilde{A}_c=\begin{pmatrix}\widetilde{c_{1}}& \widetilde{c_{2}}& \cdots&\widetilde{c_{m}}\end{pmatrix}=\begin{pmatrix}\mathcal R{c_{1}}& \mathcal R{c_{2}}& \cdots&\mathcal R{c_{m}}\end{pmatrix}.
\end{displaymath} Let us denote by $\mathcal{R}_rA:=\widetilde A_r$ and $\mathcal{R}_cA:=\widetilde A_c$, where $\mathcal{R}_r$ and  $\mathcal{R}_c$ for suitability will be called \emph{$r$-Reversing} and \emph{$c$-Reversing} respectively. Owing to $\widetilde{v}_i=v_i\widetilde{I}_m$ and $\widetilde{c}_i=\widetilde{I}_nc_j$ for $1\leq i\leq n$, $1\leq j\leq m$ we obtain that $\widetilde A_r=A\widetilde I_m$ and $\widetilde A_c=\widetilde I_n A$. Therefore $\mathcal{R}_rA=A\widetilde I_m$ and $\mathcal{R}_cA=\widetilde I_n A$ and for instance we can define palindromicity and antipalindromicity by rows and columns respectively.

\noindent Now, we can assume $A\in \mathcal{M}_{n\times m}(K)$, $B\in\mathcal M_{q\times m}(K)$ and $C\in\mathcal M_{n\times p}(K)$ given as follows.
$$\begin{array}{cc}A=\begin{pmatrix}v_{1}\\v_{2}\\ \vdots\\v_{n}\end{pmatrix}=\begin{pmatrix}f_{1}&f_{2}& \cdots&f_{m}\end{pmatrix},& v_i\in K^m,\, f_j^T\in K^n,\, 1\leq i\leq n,\,1\leq j\leq m,\\&\\B=\begin{pmatrix}s_{1}\\s_{2}\\ \vdots\\s_{q}\end{pmatrix}=\begin{pmatrix}g_{1}&g_{2}& \cdots&g_{m}\end{pmatrix},& s_i\in K^m,\, g_j^T\in K^q,\,1\leq i\leq q,\,1\leq j\leq m,\\& \\C=\begin{pmatrix}w_{1}\\w_{2}\\ \vdots\\w_{n}\end{pmatrix}=\begin{pmatrix}h_{1}&h_{2}& \cdots&h_{p}\end{pmatrix},& w_i\in K^p,\, h_j^T\in K^n,\,1\leq i\leq n,\,1\leq j\leq p.\end{array}$$
As we can see, in agreement with Section \ref{sectionv}, we transformed the column vectors $f_j$, $g_j$ and $h_j$ in the form of row vectors through the transposition of matrices ($f_j^T$, $g_j^T$ and $h_j^T$ are row vectors). Thus, we can define both \emph{Pasting by rows} (denoted by $\diamond_r$) over the matrices $A$ and $C$ and \emph{Pasting by columns} (denoted by $\diamond_c$) over the matrices $A$ and $B$ as follows.
 \begin{displaymath} A\diamond_r C=\begin{pmatrix}z_{1}\\z_{2}\\ \vdots\\z_{n}\end{pmatrix},\quad z_i=v_i\diamond w_i,\quad  A\diamond_c B=\begin{pmatrix}y_{1}&y_{2}& \cdots&y_{n}\end{pmatrix},\quad y_i^T=f_i^T\diamond g_i^T.\end{displaymath} From now on we paste column vectors directly without the use of trasposition of vectors. Thus, Pasting of column vectors $f_i$ and $g_i$ is $f_i\diamond g_i$. For instance, $f_i^T\diamond g_i^T=(f_i\diamond g_i)^T$.

\begin{proposition}\label{thmpr2} Consider matrices $A$, $B$ and $C$ under the previous assumptions. The following statements hold.
\begin{enumerate}
    \item $\mathcal R_r^2 A=A$, $\mathcal R_c^2 A=A$.

    \item $\mathcal R_r(A\diamond_r B)=(\mathcal R_r B)\diamond_r (\mathcal R_r A)$, $\mathcal R_c(A\diamond_c B)=(\mathcal R_c B)\diamond_c (\mathcal R_c A)$.
    \item $(A\diamond_r B)\diamond_r C=A\diamond_r (B\diamond_r C)$, $(A\diamond_c B)\diamond_c C=A\diamond_c (B\diamond_c C)$.
    \item $\mathcal R_r(\alpha A+\beta B)=\alpha\mathcal R_r A+\beta \mathcal R_rB$, $\mathcal R_c(\alpha A+\beta B)=\alpha \mathcal R_c A+\beta\mathcal R_cB$, $\alpha,\beta\in K$.
\item If $V=\mathcal M_{n\times m}(K)$ and $W=\mathcal M_{n\times p}(K)$, then $V \diamond W= \mathcal M_{n\times(m+p)}(K)$. In the same way, if $T=\mathcal M_{n\times m}(K)$ and $S=\mathcal M_{l\times m}(K)$, then $T \diamond S= \mathcal M_{(n+l)\times m}(K)$.

\item Let $W^r_p$ and $W^c_p$ be the set of palindromic matrices by rows and columns of $\mathcal M_{n\times m}(K)$ respectively, then the sets $W^r_p$ and $W^c_p$ are vector subspaces of $\mathcal M_{n\times m}(K)$.
\item $\dim W^r_p=n\left\lceil \frac{m}{2}\right\rceil$, $\dim W^c_p=m\left\lceil \frac{n}{2}\right\rceil$.
\item Let  $W^r_a$ and $W^c_a$ be the set of antipalindromic matrices by rows and columns of $\mathcal M_{n\times m}(K)$ respectively, then  then $W^r_a$ and $W^c_a$ are vector subspaces of $\mathcal M_{n\times m}(K)$
\item $\dim W^r_a=n\left\lfloor \frac{m}{2}\right\rfloor$, $\dim W^c_a=m\left\lfloor \frac{n}{2}\right\rfloor$
\item The sum of two palindromic matrices by rows (resp. by columns) of the same vector space is a palindromic matrix by rows (resp. by columns).
\item The sum of two antipalindromic matrices by rows (resp. by columns) of the same vector space is an antipalindromic matrix by rows (resp. by columns).
\item $\mathcal M_{n\times m}(K)=W^r_p\oplus W^r_a=W^c_p\oplus W^c_a$.
\item $\forall A\in\mathcal M_{n\times m}(K), \exists (A^r_p,A^r_a,A^c_p,A^c_a)\in W^r_p\times W^r_a\times W^c_p\times W^c_a$ such that  $A=A^r_p+A^r_a=A^c_p+A^c_a$.
\item $A\diamond_rB=A((I_n\diamond_c\mathbf{0}_{(n-m)\times m})\diamond_r \mathbf{0}_{n\times p})+\mathbf{0}_{n\times m}\diamond_r B$, $A\in \mathcal{M}_{n\times m}(K)$, $B\in\mathcal M_{n\times p}(K)$, \\ $A\diamond_cB=A((I_n\diamond_r\mathbf{0}_{n\times (m-q)})\diamond_c \mathbf{0}_{n\times p}))+\mathbf{0}_{n\times m}\diamond_c B$, $A\in \mathcal{M}_{n\times m}(K)$, $B\in\mathcal M_{p\times m}(K)$.

\end{enumerate}
\end{proposition}
\begin{proof} From (1) to (13) we proceed as in the proofs of Section \ref{sectionv} using the properties of $\mathcal R$. (14) is consequence of the definition of Pasting by rows and columns.
\end{proof}
\begin{remark} There are a lot of mathematical software in where Pasting of matrices is very easy, for example, in Matlab Pasting by rows is very easy: $[A,B]$, as well by columns $[A;B]$, however we can build our own program using our approach given in the previous proposition, following the same structure of Pasting of polynomials as in \cite{acchro1}. Thus,  we paste matrices by rows and columns using the item (14) in Proposition \ref{thmpr2}. The interested reader can proof the statements of this paper concerning to Pasting using such equations.
\end{remark}
The following proposition summarizes some properties derived from Pasting and Reversing by rows and columns with respect to classical matrix operations.\\

\begin{proposition}
The following statements hold.
\begin{enumerate}
\item $(\mathcal R_rA)^T=\mathcal R_c(A^T)$, $(\mathcal R_c A)^T=\mathcal R_r(A^T)$.
\item $(A\diamond_c B)^T=A^T\diamond_r B^T$, $(A\diamond_r B)^T=A^T\diamond_c B^T$.
\item $\mathcal R_r(AB)=A(\mathcal R_rB)$, $\mathcal R_c(AB)=(\mathcal R_cA)B$
\item $\det(\mathcal R_rA)=\det(\mathcal R_cA)=(-1)^{\left\lfloor\frac{n}{2}\right\rfloor}\det A$.
\item $(\mathcal R_c(A))^{-1}=\mathcal{R}_r(A^{-1})$, $(\mathcal R_r(A))^{-1}=\mathcal{R}_c(A^{-1})$, $\det A\neq 0$.
\item The product of two palindromic matrices by rows (resp. by columns) is a palindromic matrix by rows (resp. by columns).
\item The product of two antipalindromic matrices by rows (resp. by columns) is an antipalindromic matrix by rows (resp. by columns).
\item $AB\neq \mathbf{0}$ is a palindromic matrix by rows (resp. $AB\neq \mathbf{0}$ is a palindromic matrix by columns) if and only if $B$ is a palindromic matrix by rows (resp. $A$ is a palindromic matrix by columns).
\item $AB\neq \mathbf{0}$ is an antipalindromic matrix by rows (resp. $AB\neq \mathbf{0}$ is an antipalindromic matrix by columns) if and only if $B$ is an antipalindromic matrix by rows (resp. $A$ is an antipalindromic matrix by columns).
\end{enumerate}
\end{proposition}

\begin{proof}
We proceed according to each item
\begin{enumerate}
\item We see that $\mathcal R_c A=A\widetilde I_n=({\widetilde I_n}^TA^T)^T=(\widetilde I_nA^T)^T=(\mathcal R_rA^T)^T$ and for instance $\mathcal R_rA^T=(\mathcal R_c A)^T$. In the same way, we see that $\mathcal R_r A=\widetilde I_nA=(A^T{\widetilde I_n}^T)^T=(A^T\widetilde I_n)^T=(\mathcal R_cA^T)^T$ and for instance $\mathcal R_cA^T=(\mathcal R_r A)^T$.
\item Assume $A\in\mathcal M_{n\times m}(K)$ and $B\in \mathcal M_{n\times p}(K)$. Let $v_i$ and $w_i$ be the row vectors of $A$ and $B$ respectively. Thus $v_i\diamond w_i$, $i=1,\ldots,n$, are the row vectors of $A\diamond_rB$, then $(v_i\diamond w_i)^T=v_i^T\diamond w_i^T$ are the column vectors of $(A\diamond_rB)^T=A^T\diamond_cB^T$. In the same way we can assume $A\in\mathcal M_{n\times m}(K)$ and $B\in \mathcal M_{p\times m}(K)$. Let $c_j$ and $d_j$ be the column vectors of $A$ and $B$ respectively. Therefore $c_j\diamond d_j$, $j=1,\ldots,m$, are the column vectors of $A\diamond_cB$, then $(c_j\diamond d_j)^T=c_j^T\diamond d_j^T$ are the row vectors of $(A\diamond_cB)^T=A^T\diamond_rB^T$.
\item  Consider $A\in\mathcal M_{n\times m}(K)$ and $B\in \mathcal M_{m\times p}(K)$, therefore $\mathcal R_r(AB)=\widetilde I_n(AB)=(\widetilde I_nA)B=(\mathcal R_rA)B$. In the same way, $\mathcal R_c(AB)=(AB)\widetilde I_p=A(B\widetilde I_p)=A(\mathcal R_cB)$.
\item  Consider $A\in\mathcal M_{n\times n}(K)$, for instance we obtain $\det(\mathcal R_rA)=\det(\widetilde I_nA)=\det(\widetilde I_n)\det A=\det (A)\det(\widetilde I_n)=\det(A\widetilde I_n)=\det(\mathcal R_cA)$. Now, it is follows by induction that we can transform $\widetilde I_n$ into $I_n$ throught $\lfloor\frac{n}{2}\rfloor$ elementary operations, therefore, $\det(\widetilde I_n)=(-1)^{\lfloor\frac{n}{2}\rfloor}$.
\item Assume $A\in\mathcal M_{n\times n}(K)$, being $\det A\neq 0$. For instance, $(\mathcal R_cA)^{-1}=(A\widetilde I_n)^{-1}=\widetilde I_n^{-1}A^{-1}=\widetilde I_nA^{-1}=\mathcal R_r(A^{-1})$. In the same way $(\mathcal R_rA)^{-1}=(\widetilde I_nA)^{-1}=A^{-1}\widetilde I_n^{-1}=A^{-1}\widetilde I_n=\mathcal R_c(A^{-1})$.
\item Assume $A\in\mathcal M_{n\times m}(K)$ and $B\in \mathcal M_{m\times p}(K)$, such that $\mathcal R_rA=A$ and $\mathcal R_rB=B$, therefore $\mathcal R_r(AB)=(\mathcal R_rA)B=AB$. In the same way, asumming $\mathcal R_cA=A$ and $\mathcal R_cB=B$ we obtain $\mathcal R_c(AB)=A\mathcal R_cB=AB$.
\item Assume $A\in\mathcal M_{n\times m}(K)$ and $B\in \mathcal M_{m\times p}(K)$ such that $\mathcal R_rA=-A$ and $\mathcal R_rB=-B$, therefore $\mathcal R_r(AB)=(\mathcal R_rA)B=-AB$. In the same way, asumming $\mathcal R_cA=-A$ and $\mathcal R_cB=-B$ we obtain $\mathcal R_c(AB)=A\mathcal R_cB=-AB$.
\item Assume $A\in\mathcal M_{n\times m}(K)$ and $B\in \mathcal M_{m\times p}(K)$, being $AB\neq \mathbf{0}$. By previous item we see that if $A$ is palindromic by rows (resp. if $B$ is palindromic by columns), then $AB$ is palindromic by rows (resp. then $AB$ is palindromic by columns). Now, suppose that $\mathcal R_r(AB)=AB\neq\mathbf{0}$, for instance $(\mathcal R_rA)B=(\widetilde I_nA)B=AB$, which implies that $\mathcal R_rA=A$. In the same way, if we asumme $\mathcal R_c(AB)=AB\neq\mathbf{0}$, for instance $(\mathcal R_cA)B=A(\widetilde I_nB)=AB\neq\mathbf{0}$, which implies that $\mathcal R_cB=B$.
\item Assume $A\in\mathcal M_{n\times m}(K)$ and $B\in \mathcal M_{m\times p}(K)$, being $AB\neq \mathbf{0}$. By previous item we see that if $A$ is antipalindromic by rows (resp. if $B$ is antipalindromic by columns), then $AB$ is antipalindromic by rows (resp. then $AB$ is antipalindromic by columns). Now, suppose that $\mathcal R_r(AB)=-AB\neq\mathbf{0}$, for instance $(\mathcal R_rA)B=(\widetilde I_nA)B=-AB$, which implies that $\mathcal R_rA=-A$. In the same way, if we asumme $\mathcal R_c(AB)=-AB\neq\mathbf{0}$, for instance $(\mathcal R_cA)B=A(\widetilde I_nB)=-AB\neq\mathbf{0}$, which implies that $\mathcal R_cB=-B$.
\end{enumerate}
\end{proof}
Now, in a natural way, we can introduce the sets $$\mathsf{W}_{pp}:=W_p^r\cap W_p^c,\, \mathsf{W}_{pa}:=W_p^r\cap W_a^c,\, \mathsf{W}_{ap}:=W_a^r\cap W_p^c\,\, \mathrm{and}\,\, \mathsf{W}_{aa}:=W_a^r\cap W_a^c.$$ The sets $\mathsf{W}_{pp}$ and $\mathsf{W}_{aa}$ correspond to the set of \emph{double palindromic matrices} and the set of \emph{double antipalindromic matrices} respectively.

\begin{proposition}
The following statements hold.
\begin{enumerate}
\item $\mathsf{W}_{pp}$, $\mathsf{W}_{pa}$, $\mathsf{W}_{ap}$ and $\mathsf{W}_{aa}$ are vector subspaces of $\mathcal M_{n\times m}(K)$.
\item $\dim \mathsf W_{pp}=\left\lceil \frac{n}{2}\right\rceil\left\lceil \frac{m}{2}\right\rceil$, $\dim \mathsf W_{pa}=\left\lceil \frac{n}{2}\right\rceil\left\lfloor \frac{m}{2}\right\rfloor$, $\dim \mathsf W_{ap}=\left\lfloor \frac{n}{2}\right\rfloor\left\lceil \frac{m}{2}\right\rceil$ and $\dim \mathsf W_{aa}=\left\lfloor \frac{n}{2}\right\rfloor\left\lfloor \frac{m}{2}\right\rfloor$,
\item $\mathcal M_{n\times m}(K)=\mathsf W_{pp}\oplus \mathsf W_{pa}\oplus\mathsf W_{ap}\oplus \mathsf W_{aa}$.
\item $\forall A\in\mathcal M_{n\times m}(K), \exists (A_{pp},A_{pa},A_{ap},A_{aa})\in \mathsf{W}_{pp}\times\mathsf{W}_{pa}\times\mathsf{W}_{ap}\times\mathsf{W}_{aa}$ such that $A=A_{pp}+A_{pa}+A_{ap}+A_{aa}$.
\end{enumerate}
\end{proposition}

\begin{proof} The intersection of vector subspaces is a vector subspace. The rest is obtained through elementary properties of vector subspaces and by the nature of $\mathsf{W}_{pp}$, $\mathsf{W}_{pa}$, $\mathsf{W}_{ap}$ and $\mathsf{W}_{aa}$.
\end{proof}

Finally, according to Remark \ref{remacarnu}, we present the results concerning to the relationship between Reversing and the generalized vector product of $n-1$ vectors of $K^n$, which are given in \cite[\S 3]{acarnu}.

Let $M_{1}=\left(
m_{11},m_{12},\ldots ,m_{1n}\right) ,\ldots ,M_{n-1}=\left( m_{\left(
n-1\right) ,1},a_{\left( n-1\right) ,2},\ldots ,m_{\left( n-1\right)
,n}\right)$, be $n-1$ vectors belonging to $K^{n}$. It is known that the generalized vector product of these vectors is given by
$$\times \left( M_{1},\text{ }M_{2},\text{ }\ldots ,\text{ }M_{n-1}\right)
=\sum_{k=1}^{n}\left(
-1\right) ^{1+k}\det \left( M^{(k)}\right) e_{k},
$$
being $e_{k}$ the $k$-th element of the canonical basis for $K^{n}$ and $M^{\left( k\right) }$ is the square matrix obtained after the deleting of the $k$-th column of the matrix $M=\left( m_{ij}\right)\in\mathcal M_{\left(
n-1\right) \times n}(K)$, for more information see \cite{acarnu,Lang}. Therefore, the matrix $M^{\left( k\right) }$ is a square matrix of size $\left(
n-1\right) \times \left( n-1\right) $ and is given by
$$
M^{(k)}=\left( m_{i,j}^{(k)}\right)=\left\{
\begin{array}{l}
\left(m_{i,j}\right)\text{, whether }j<k \\
\left(m_{i,j+1}\right)\text{, whether }j\geq k%
\end{array}
\right.,\quad M=\begin{pmatrix}
M_1\\M_2\\ \vdots\\ M_{n-1}
\end{pmatrix} .
$$

\begin{proposition}\label{prop2}
Consider the matrix $M=\left( m_{ij}\right)\in\mathcal M_{\left(
n-1\right) \times n}(K)$, then $\mathcal R_r(M^{\left( k\right) })=M^{\left( n-k+1\right) }\widetilde I_{n-1}$, for $1\leq k\leq n$.
\end{proposition}

\begin{proof}
We know that $\mathcal R_rM=M\widetilde I_{n}=\left( m_{i,n-j+1}\right)$,  $1\leq j\leq n$. Therefore
\begin{eqnarray*}
\mathcal R_rM^{(k)} =\mathcal R_r\left(m_{i,j}^{(k)}\right)
=\left\{
\begin{array}{l}
\left(m_{i,n-j+1}\right)\text{, whether }j<k, \\
\left(m_{i,n-(j+1)+1}\right)\text{, whether }j\geq k
\end{array}%
\right. .
\end{eqnarray*}
On the other hand,
\begin{eqnarray*}
M^{\left( n-k+1\right) }\widetilde I_{n-1}=\left( m_{i,j}^{\left( n-k+1\right) }\right)\widetilde I_{n-1}
=\left\{
\begin{array}{l}
\left(m_{i,j}\right)\widetilde I_{n-1}\text{, whether }j<n-k+1 \\
\left(m_{i,j+1}\right)\widetilde I_{n-1}\text{, whether }j\geq n-k+1%
\end{array}%
\right.
\\ \\
=\left( m_{i,(n-1)-j+1}^{\left(
n-k+1\right) }\right) =\left( m_{i,n-j}^{\left( n-k+1\right) }\right) =\left\{
\begin{array}{l}
\left(m_{i,\left( n-j\right) }\right)\text{, whether }n-j<n-k+1 \\
\left(m_{i,\left( n-j\right) +1}\right)\text{, whether }n-j\geq n-k+1%
\end{array}%
\right.
\\ \\=\left\{
\begin{array}{l}
\left(m_{i,\left( n-j\right) }\right)\text{, whether }j>k-1 \\
\left(m_{i,\left( n-j\right) +1}\right)\text{, whether }j\leq k-1%
\end{array}
\right.
=\left\{
\begin{array}{l}
\left(m_{i,n-j}\right)\text{, whether }j\geq k \\
\left(m_{i,n-j+1}\right)\text{, whether }j<k%
\end{array}
\right.,
\end{eqnarray*}
for instance $\mathcal R_r(M^{\left( k\right) })=M^{\left( n-k+1\right) }\widetilde I_{n-1}$, for $1\leq k\leq n$.
\end{proof}

\begin{proposition}\label{prop3} Consider the vectors $M_i\in K^n$, where $1\leq i \leq n-1$. Then $$\times \left(\mathcal R_rM_1,\mathcal R_rM_{2},\ldots,\mathcal R_rM_{n-1}\right)=(-1)^{\left\lceil\frac{3n}{2}\right\rceil}\mathcal R_r\left(\times (M_{1},M_{2},\ldots,M_{n-1})\right).$$
\end{proposition}

\begin{proof}
For suitability we write $$\mathfrak{M}=\times \left( \mathcal R_rM_{1},\mathcal R_rM_2,
\ldots,\mathcal R_rM_{n-1}\right).$$ Now, applying the generalized vector product we obtain

\begin{eqnarray*}
\mathfrak{M}& =&\sum_{k=1}^{n}(-1)^{k+1}\det \left(\mathcal R_r(M^{(k)}\right) e_{k}\\
& =&\sum_{k=1}^{n}\left( -1\right) ^{k+1}\det \left( M^{\left( n-k+1\right)}\widetilde I_{n-1}\right) e_{k}\\ &=&\sum_{k=1}^{n}\left( -1\right) ^{k+1}\det \left( M^{\left( n-k+1\right)}\right) \det \left( \widetilde I_{n-1}\right) e_{k}\\
& =&\det \left( \widetilde I_{n-1}\right) \sum_{k=1}^{n}\left( -1\right) ^{n-k}\det \left( M^{\left( k\right) }\right) e_{n-k+1} \\&=&\left( -1\right) ^{n+1}\det \left( \widetilde I_{n-1}\right) \sum_{k=1}^{n}\left(
-1\right) ^{k+1}\det \left( M^{\left( k\right) }\right) e_{n-k+1}\\
&=&\left( -1\right) ^{n+1}\det \left( \widetilde I_{n-1}\right) \left(
\sum_{k=1}^{n}\left( -1\right) ^{k+1}\det \left( M^{\left( k\right) }\right)
e_{k}\right) \widetilde I_{n}\\ &=&\left( -1\right) ^{n+1}\det \left(\widetilde I_{n-1}\right) \mathcal R_r\left( {%
\times \left( M_{1},\text{ }M_{2},\text{ }\ldots ,\text{ }M_{n-1}\right) }
\right)
\end{eqnarray*}

and for instance,%
\begin{eqnarray*}
\mathfrak{M}  &=&(-1)^{n+1}(-1)^{\left\lfloor\frac{n-1}{2}\right\rfloor} \mathcal R_r\left( {%
\times \left( M_{1},\text{ }M_{2},\text{ }\ldots ,\text{ }M_{n-1}\right) }
\right)\\
&=&\left\{
\begin{array}{l}
\left( -1\right) ^{\frac{3n}{2}}\mathcal R_r\left( {%
\times \left( M_{1},\text{ }M_{2},\text{ }\ldots ,\text{ }M_{n-1}\right) }
\right),\, n=2k\\
\left( -1\right) ^{\frac{3n+1}{2}}\mathcal R_r\left( {%
\times \left( M_{1},\text{ }M_{2},\text{ }\ldots ,\text{ }M_{n-1}\right) }
\right), n=2k-1
\end{array}%
\right. \\
&=&(-1)^{\left\lfloor\frac{3n}{2}\right\rfloor} \mathcal R_r\left( {%
\times \left( M_{1},\text{ }M_{2},\text{ }\ldots ,\text{ }M_{n-1}\right) }
\right).
\end{eqnarray*} Thus we conclude the proof.
\end{proof}
\begin{remark}
If $M$ is a palindromic matrix by rows, then the minors $M^{\left( k\right) }$ have at least $\left\lfloor\dfrac{n}{2}\right\rfloor-1$ pair of equal columns. This implies that for  $n\geq 4$, the minors have at least one pair of equal columns and for instance $\det \left( M^{\left(
k\right) }\right) =0$ for all $1\leq k\leq n$, which lead us to
$$
\times \left( M_{1},\text{ }M_{2},\text{ }\ldots ,\text{ }M_{n-1}\right)
=\mathbf{0}\in K^{n}.
$$

\noindent This means that the generalized vector product of $(n-1)$ palindromic vectors belonging to  $K^{n}$ is interesting whenever $1\leq n\leq 3$. The same result is obtained when we assume $M$ as an antipalindromic matrix by rows, so we recover the previous results given in Section \ref{sectionv} with respect to the vector product. Moreover, some rows of $M$ can be palindromic vectors, while the rest can be antipalindromic vectors, in this way, we can obtain similar results.
\end{remark}

\subsection{Pasting and Reversing simultaneously by rows and columns} As in previous sections, following Section \ref{sectionv}, we can consider the matrices $$A=(v_{11},\ldots,v_{1m},\ldots,v_{n1},\ldots, v_{nm}),\quad B=(w_{11},\ldots,w_{1q},\ldots,w_{p1},\ldots w_{pq})$$ as vectors. To avoid confusion in this section, we use $\widehat{A}$ instead of $\widetilde A$ to denote Reversing of $A$. Thus, we can see, in a natural way, that $$\mathcal R A=\widehat{A}=A\widehat I_{nm}=(v_{nm},\ldots,v_{n1},\ldots,v_{1m},\ldots, v_{11})$$ and also for $n=p$ or $m=q$ (exclusively) that $$A\diamond B=(v_{11},\ldots,v_{1m},\ldots,v_{n1},\ldots v_{nm},w_{11},\ldots,w_{1q},\ldots,w_{p1},\ldots, w_{pq}).$$ We come back to express $\mathcal{R}A$ and $A\diamond B$ in term of matrices instead of vectors, i.e.,
$$\mathcal{R}A=\begin{pmatrix}
v_{nm}&\ldots&v_{n1}\\\vdots\\v_{1m}&\ldots& v_{11}
\end{pmatrix},\, A\diamond B=\left\lbrace\begin{array}{l}
\begin{pmatrix}
v_{11}&\ldots&v_{1m}&w_{11}&\ldots&w_{1q}\\ \vdots\\v_{n1}&\ldots& v_{nm}&w_{p1}&\ldots&w_{pq}
\end{pmatrix},\, \begin{matrix}n=p\\m\neq q\end{matrix} \\ \\\begin{pmatrix}
v_{11}&\ldots&v_{1m}\\ \vdots\\v_{n1}&\ldots& v_{nm}\\w_{11}&\ldots&w_{1p}\\ \vdots\\w_{p1}&\ldots&w_{pq}
\end{pmatrix},\,\begin{matrix}n\neq p\\m= q\end{matrix}\end{array}\right.$$ We say that any matrix $P$, with the conditions established above, is a palindromic matrix whether $\widehat{P}=\mathcal RP=P$. In the same way, we say that any matrix $A$, with the conditions established above, is an antipalindromic matrix whether $\widehat{A}=\mathcal RA=-A$.
Note that $\mathcal R(I_n)= \widehat{I}_n$ where $I_n$ is written in the matrix form, but it should keep in mind that $\mathcal R(I_n)=I_n\widehat{I}_{n^2}$, where $I_n$ is written as vector. Thus, we arrive to the following elementary result.
\begin{lemma}\label{lemmkey}
Consider $M\in\mathcal{M}_{n\times m}(K)$. Then
$$\mathcal RA=\begin{pmatrix}\mathcal R{v_n}\\\vdots\\ \mathcal R{v_1}\end{pmatrix},\quad \textrm{where}\quad A=\begin{pmatrix} {v_1}\\\vdots\\ {v_n}\end{pmatrix}.$$
\end{lemma}
\begin{proof}
It is followed by definition of Reversing in matrices seen as vectors.
\end{proof}
The following proposition summarizes the previous results for matrices as vectors.

\begin{proposition}\label{thmpr1} Consider $A \in M_{n\times m}(K)$, $B \in M_{p\times q}(K)$ and $C \in M_{r\times s}(K)$ satisfying the conditions established above, the following statements hold.
\begin{enumerate}
    \item $\mathcal R^2A=A$
    \item $\mathcal R(A\diamond B)=\mathcal R(B)\diamond \mathcal R(A)$
    \item $(A\diamond B)\diamond C=A\diamond (B\diamond C)$
    \item $\mathcal R(bA+cB)=b\mathcal R(A)+c\mathcal R(B)$ where $b,c\in K$ and $A,B\in\mathcal M_{n\times m}(K)$.
\item If $V=\mathcal M_{n\times m}(K)$ and $W=\mathcal M_{p\times q}(K)$, then $V \diamond W= \mathcal M_{r\times s}(K)$, where either $n=p=r$, $m\neq q$ and $s=m+q$, or $n\neq p$, $m=q$ and $r=n+p$.

\item Let $W_p$ be the set of palindromic matrices of $\mathcal M_{n\times m}(K)$, then $W_p$ is a vector subspace of $\mathcal M_{n\times m}(K)$.
\item $\dim W_p=\left\lceil \frac{nm}{2}\right\rceil$
\item Let $W_a$ be the set of antipalindromic matrices of $\mathcal M_{n\times m}(K)$, then  $W_a$ is a vector subspace of $\mathcal M_{n\times m}(K)$.
\item $\dim W_a=\left\lfloor \frac{nm}{2}\right\rfloor$
\item The sum of two palindromic matrices of the same vector space is a palindromic matrix.
\item The sum of two antipalindromic matrices of the same vector space is an antipalindromic matrix.
\item $\mathcal M_{n\times m}(K)=W_p\oplus W_a$.
\item $\forall A\in\mathcal M_{n\times m}(K), \exists (A_p,A_a)\in W_p\times W_a$ such that  $A=A_p+A_a$.
\end{enumerate}
\end{proposition}
\begin{proof} Proceed as in the proofs of Section \ref{sectionv} using Lemma \ref{lemmkey}.
\end{proof}
The following result shows the relationship of Reversing with matrices classical operations.

\begin{proposition}\label{prop41} The following statements hold.
\begin{enumerate}
\item $\mathcal R(I_n)=I_n$
\item $\mathcal R=\mathcal R_c \mathcal R_r=\mathcal R_r\mathcal R_c$.
\item $\mathcal R(AB)=\mathcal{R}(A)\mathcal R (B)$.
\item $(\mathcal R(A))^{-1}=\mathcal{R}(A^{-1})$, $\det A\neq 0$.
\item $\det (\mathcal{R}(A))=\det A$.
\item $\mathrm{Tr}(\mathcal{R}(A))=\mathrm{Tr} A$.
\item $\mathcal{R}(A^T)=(\mathcal{R}(A))^T$.
\item The product of two palindromic matrices is a palindromic matrix.
\item The product of two antipalindromic matrices is a palindromic matrix.
\item The product of one palindromic matrix with one antipalindromic matrix is an antipalindromic matrix.
\end{enumerate}
\end{proposition}

\begin{proof}
We prove each item.
\begin{enumerate}
\item Due to $I_n$ can be seen as the vector $$(1,0,\ldots,0,0,1,\ldots,0,\ldots, 0,\ldots 0,1)\in K^{n^2},$$ then $$\mathcal R(I_n)=I_n\widehat I_{n^2}=(1,0,\ldots,0,0,1,\ldots,0,\ldots, 0,\ldots 0,1)=I_n.$$
\item Assume $A\in \mathcal M_{n\times m}(K)$ and $B\in \mathcal M_{m\times r}(K)$. Thus, $AB\in \mathcal M_{n\times r}(K)$, $\mathcal R A=\widehat A\in \mathcal M_{n\times m}$, $\mathcal R B=\widehat B\in \mathcal M_{m\times r}$ and $\mathcal R (AB)=\widehat{AB}\in \mathcal M_{n\times r}$. Suppose that $A=[a_{ij}]_{n\times m}$, $B=[b_{ij}]_{m\times r}$, $AB=C=[c_{ij}]_{n\times r}$ and $\widehat C=[d_{ij}]_{n\times r}$, for instance  $$c_{ij}=\sum_{k=1}^ma_{ik}b_{kj},\quad d_{ij}=c_{(n+1-i)(r+1-j)}=\sum_{k=1}^ma_{(n+1-i)k}b_{k(r+1-j)},$$ which implies that $\widehat C=\widehat A \widehat B$ and then  $\mathcal R(AB)=\mathcal{R}(A)\mathcal R (B)$.
\item Assume $A\in \mathcal M_{n\times n}(K)$, being $\det A\neq 0$. Therefore $$\mathcal R(AA^{-1})=\mathcal R(A)\mathcal R (A^{-1})=\mathcal R(I_n)=I_n.$$ For instance, $(\mathcal R(A))^{-1}=\mathcal{R}(A^{-1})$.
\item Assume $A\in \mathcal M_{n\times n}(K)$. Due to $\mathcal R(A)$ is obtained throughout $2k$ elementary operations of $A$, interchanging $k$ rows and interchanging $k$ columns, then $\det(\mathcal R(A))=(-1)^{2k}\det A=\det A$.
\item  Assume $A=[a_{ij}]_{n\times n}\in \mathcal M_{n\times n}(K)$. Thus $$\mathrm{Tr}\widehat A=\sum_{i=1}^na_{(n+1-i)(n+1-i)}=a_{nn}+a_{(n-1)(n-1)}+\cdots+ a_{22}+a_{11}=\sum_{k=1}^na_{kk}=\mathrm{Tr} A.$$
\item Assume $A=[a_{ij}]_{n\times m}\in \mathcal M_{n\times m}(K)$ and $\widehat{A^T}=[d_{ij}]_{m\times n}\in \mathcal M_{m\times n}(K)$. We see that $d_{ij}=a_{(m+1-j)(n+1-i)}=c_{ji}$, where $(\widehat A)^T=[c_{ji}]_{m\times n}$ and for instance $\widehat{A^T}=(\widehat A)^T$, which implies that $\mathcal{R}(A^T)=(\mathcal{R}(A))^T$.
\item Assume $\widehat A=A$ and $\widehat B=B$. Therefore, $\widehat{AB}=\widehat A\widehat B=AB$.
\item Assume $\widehat A=-A$ and $\widehat B=-B$. Therefore, $\widehat{AB}=\widehat A\widehat B=AB$.
\item Assume $\widehat A=-A$ and $\widehat B=B$. Therefore, $\widehat{AB}=\widehat A\widehat B=-AB$.
\end{enumerate}
\end{proof}
At this point, we have considered Pasting over an special case of matrices. As we can see, it can be confused when both matrices have the same size, how can we paste them? Another natural question is: how can we paste to matrices whenever $n\neq p$ and $m\neq q$? To avoid this difficulty we introduce \emph{Pasting by blocks}, which will be denoted by $\diamond_b$. Consider matrices $A\in\mathcal{M}_{n\times m}(K)$ and $B\in\mathcal{M}_{r\times s}(K)$, Pasting by blocks of $A$ with $B$ is given by $$A\diamond_b B:=\begin{pmatrix}A & \mathbf{0}_{n\times s}\\\mathbf{0}_{r\times m} &B \end{pmatrix}\in\mathcal{M}_{(n+r)\times(m+s)}(K).$$ It is well known that Pasting by blocks corresponds to a particular case of \emph{block matrices}, also called \emph{partitioned matrices}, see \cite{hall,lantis}.
The following result, although is known from block matrices point of view, is consequence of Proposition \ref{thmpr1} and Proposition \ref{prop41} considering $\mathcal R$ as above and $\diamond_b$ instead of $\diamond$.

\begin{proposition}
Consider the matrices $A \in M_{n\times m}(K)$, $B \in M_{p\times q}(K)$ and $C \in M_{r\times s}(K)$, the following statements hold.
\begin{enumerate}
 \item $\mathcal R(A\diamond_b B)=\mathcal R(B)\diamond_b \mathcal R(A)$
    \item $(A\diamond_b B)\diamond_b C=A\diamond_b (B\diamond_b C)$
\item If $V=\mathcal M_{n\times m}(K)$ and $W=\mathcal M_{p\times q}(K)$, then $V \diamond_b W= \mathcal M_{r\times s}(K)$, where $r=n+p$ and $s=m+q$.
\item $(A\diamond_b B)^T=A^T\diamond_b B^T$.
\item $\det (A\diamond_b B)=\det A  \det B$.
\item $\mathrm{Tr}(A\diamond_b B)=\mathrm{Tr} A+\mathrm{Tr} B$.
\item $(A\diamond_b B)^{-1}=A^{-1}\diamond_b B^{-1}$, $\det(AB)\neq 0$.
\end{enumerate}
\end{proposition}
\begin{proof} According to each item:

(1) follow directly from the definition of Pasting by blocks, Reversing and Pasting of vectors.

(2) is due to $A \diamond_b B=A\oplus B$.

(3) to (7) are known properties of block matrices.
\end{proof}

\section*{Final Remarks}
This paper is an invitation to undergraduate students and teachers of linear algebra to explore Pasting and Reversing in advanced topics of linear algebra, as well in other branches of mathematics and physics. Here we solved one question proposed in \cite{acchro1}, which relates Pasting and Reversing with vector spaces and matrix theory. However, there are a lot of questions about the applications of Pasting and Reversing, as well generalizations of these operations. The interested reader can try to solve the following natural questions.
\begin{itemize}
\item What properties for Pasting and Reversing hold when we use $A_\sigma$ as the companion matrix of the linear transformation $\mathcal R_\sigma$, where $\sigma\in S_n$ and $\sigma^2\neq e$? In particular, what happens when $\mathcal R_\sigma^n=e$, being $\mathcal R_\sigma^{n-1}\neq e$?
\item How can we apply Pasting and Reversing in mathematical physics and dynamical systems? It is easy to see that Pauli spin matrices can be written in term of Reversing and the symplectic matrix can be written in term of Pasting. Furthermore, for reversible systems this approch has been applied succesfully, see \cite{MaRa}.
\end{itemize}

\section*{Acknowledgements}
The authors thank to David Blazquez-Sanz by their useful comments and suggestions on this work as well to Greisy Morillo by their hospitality and support during the final part of this work. The first author is partially supported by Dacobian project from UPC -- Spain and by Agenda DIDI 2012--2013 Uninorte.

\end{document}